\documentclass[a4paper,12pt]{article}

\usepackage[utf8]{inputenc}
\usepackage{comment}
\usepackage{amstext}
\usepackage{amsfonts}
\usepackage{amsthm}
\usepackage{textcomp}
\usepackage{amssymb}
\usepackage{amscd}
\usepackage{amsmath}
\usepackage{xcolor}
\usepackage{tikz-cd}
\usepackage{caption}
\usepackage{enumitem}
\usepackage{theoremref}
\setlist[enumerate]{label=({\arabic*})}

\newtheorem{defn}{Definition}[section]

\newtheorem{lem}[defn]{Lemma}

\newtheorem{prop}[defn]{Proposition}

\newtheorem{manualtheoreminner}{Theorem}
\newenvironment{manualtheorem}[1]{%
  \renewcommand\themanualtheoreminner{#1}%
  \manualtheoreminner
}{\endmanualtheoreminner}

\newcommand{\N}{\mathbb{N}}
\newcommand{\R}{\mathbb{R}}
\newcommand{\Z}{\mathbb{Z}}

\newcommand{\B}{{\mathcal B}}

\newcommand{\hc}[1]{{\mathcal {HC} {#1}}}
\newcommand{\hd}[1]{{\mathcal {HD} {#1}}}
\newcommand{\gh}[1]{{\mathcal {GH} {#1}}}

\DeclareMathOperator*\lowlim{\underline{lim}}
\DeclareMathOperator*\uplim{\overline{lim}}

\begin{document}
\title{Shift-like Operators on $L^p(X)$} 

\author{Emma D'Aniello \\
\and  
Udayan B. Darji \\ 
\and 
Martina Maiuriello  \\}

\newcommand{\Addresses}{{
  \bigskip
  \footnotesize

  E.~D'Aniello, \\
  \textsc{Dipartimento di Matematica e Fisica,\\ Universit\`a degli Studi della Campania ``Luigi Vanvitelli",\\
  Viale Lincoln n. 5, 81100 Caserta, ITALIA} \\
  \textit{E-mail address: \em emma.daniello@unicampania.it} 
  \medskip

  U.B.~Darji,\\
  \textsc{Department of Mathematics,\\ University of Louisville,\\
  Louisville, KY 40292, USA}\\
  \textit{E-mail address: \em ubdarj01@louisville.edu}
  \medskip

  M.~Maiuriello,\\
  \textsc{Dipartimento di Matematica e Fisica,\\ Universit\`a degli Studi della Campania ``Luigi Vanvitelli",\\
  Viale Lincoln n. 5, 81100 Caserta, ITALIA}\\
  \textit{E-mail address: \em martina.maiuriello@unicampania.it}

}}

\date{\today}
\maketitle

%\tableofcontents
%https://www.overleaf.com/project/5ec563dfe810fb00010ab5f9#

\begin{abstract}
 In this article we develop a general technique which takes a known characterization of a property for weighted backward shifts and lifts it up to a characterization of that property for a large class of operators on $L^p(X)$. We call these operators ``shift-like''. The properties of interest include chaotic properties such as Li-Yorke chaos, hypercyclicity, frequent hypercyclicity as well as properties related to hyperbolic dynamics such as shadowing, expansivity and generalized hyperbolicity. Shift-like operators appear naturally as composition operators on $L^p(X)$ when the underlying space is a dissipative measure system. In the process of proving the main theorem, we provide some results concerning when a property is shared by a linear dynamical system and its factors. 

\end{abstract}

\let\thefootnote\relax\footnote{\date{} 
}

%##############################
\section{Introduction}
%##############################

Weighted shifts is an important class of operators in linear dynamics. It is a tool for constructing examples and counterexamples as well as inspiring new conjectures. Many dynamical properties are studied and characterized for weighted shifts before the general theory of the dynamical property in question reveals itself. 

A systematic study of composition operators on $L^p(X)$ in the setting of linear dynamics was initiated in \cite{BayartDarjiPiresJMAA2018, BernardesDarjiPiresMM2020}. The motivation for the study of composition operators is to have a concrete and large class of operators which can be utilized as examples and counterexamples in linear dynamics. The class of  composition operators includes weighted shifts as a special case. It also includes other large classes of operators,  such as composition operators induced by measures on odometers \cite{BDDPOdometers}. 

Let us briefly recall the composition operators on $L^p(X)$, for $1\leq p < \infty$. We start with a $\sigma$-finite measure space $(X,{\mathcal B},\mu)$ and a bimeasurable invertible  map $f:X \rightarrow X$ such that the Radon-Nikodym derivative of $\mu \circ f$ with respect to $\mu$ is bounded below. Then, $T_f:L^p(X) \rightarrow L^p(X)$ defined by
\[ T_f (\varphi) = \varphi \circ f
\]
is a well-defined bounded linear operator on $L^p(X)$ known as the composition operator. 

For a general $(X,{\mathcal B},\mu, f, T_f)$ as above, Bayart, Pires and the second author \cite{BayartDarjiPiresJMAA2018}  gave  necessary and sufficient conditions on $f$ that guarantee hypercyclicity and mixing of $T_f$. Bernardes, Pires and the second author \cite{BernardesDarjiPiresMM2020} gave necessary and sufficient conditions on $f$ that guarantee that $T_f$ is Li-Yorke chaotic. These are very general characterizations and, as a specific case, when $X = \Z$ and $f$ is the $+1$-map, they yield well-known characterizations of hypercyclicity, mixing and Li-Yorke chaos for weighted backward shifts.

Unfortunately, the behaviors of  frequent hypercyclicity and the shadowing property in the setting of composition operators are complicated. This is not unexpected as it took some time to complete the characterization of frequent hypercyclicity even for weighted backward shifts. Using ideas of Bayart and Ruzsa \cite{BayartRuzsa}, the second author and Pires \cite{DarjiPires2021} gave a characterization of frequently hypercyclic operators among composition operators on $L^p(X)$ when the measure system is dissipative and of bounded distortion. 

In a parallel development concerning the shadowing property, Bernardes and Messaoudi \cite{BernardesMessaoudiETDS2020} gave a characterization of the shadowing property for the class of weighted backward shifts. The current authors of this article  \cite{D'AnielloDarjiMaiuriello} lifted their result to the setting of composition operators where the measure system is dissipative and of bounded distortion. 

Motivated by these results, we develop a general method which takes a characterization of a linear dynamical property for weighted backward shifts and translates  it into the setting of composition operators on $L^p(X)$ where the underlying system $(X,{\mathcal B},\mu, f)$ is a dissipative, measurable system of bounded distortion.

Before stating the main result of the article, let us fix some notation and terminology.   
Let $(X,{\mathcal B},\mu)$ be a 
$\sigma$-finite measure space and $f:X \rightarrow X$ be a bimeasurable invertible map such that the Radon-Nikodym derivatives of $\mu \circ f$ and $\mu \circ f^{-1}$ with respect to $\mu$ are bounded below. In particular, this implies that the composition operators $T_f$ and $T_{f^{-1}}$ are well-defined, invertible, bounded linear operators with $T_f^{-1} = T_{f^{-1}}$. We say that $(X,{\mathcal B},\mu, f)$ is a {\em dissipative system generated by $W$} if $W \in \B$, $0 < \mu (W) < \infty$ and $X = \dot {\cup} _{k=-\infty}^{+ \infty} f^k (W)$, where the symbol $\dot {\cup}$ denotes pairwise disjoint union (we are slightly deviating from the usual definition of dissipative system, i.e.,  $\mu(W) < \infty$ is usually not required).  
Moreover, if there exists $K>0$ such that
\begin{equation}
 \dfrac{1}{K} \dfrac{\mu(f^k(W))}{\mu(W)}\ \leq \dfrac{\mu(f^k(B))}{\mu(B)}\ \leq K \dfrac{\mu(f^k(W))}{\mu(W)}\, \tag{$\Diamond$}
\end{equation}
for all $k \in \mathbb Z$ and  all measurable $B \subseteq W$ with $\mu(B)>0$, then we say that $(X,{\mathcal B},\mu, f)$ is a {\em dissipative system of bounded distortion generated by $W$.}    We will call $(X,{\mathcal B},\mu, f, T_f)$ {\em a dissipative composition dynamical system of bounded distortion, generated by 
set $W$} \cite[Definition 2.6.3]{D'AnielloDarjiMaiuriello}.
It is easy to see a connection between such a system and weighted backward shifts. Namely, if we let $B_{w}$ be the bilateral weighted backward shift on $\ell ^p(\Z)$, with the weight sequence \begin{equation} \label{wk}
    w_{k} = {\left(\frac{\mu(f^{k-1}(W))}{\mu(f^{k}(W))}\right)}^{\frac{1}{p}},
    \end{equation}
then $B_w$ is a linear factor of $T_f$ in a natural way. More specifically,
$\Gamma:L^p(X) \rightarrow \ell^p(\Z)$, defined by
\begin{equation}
  \Gamma (\varphi) (k)  = \dfrac{\mu(f^{k}(W))^{\frac{1}{p}}}{\mu(W)} \int_{W}  \varphi \circ f^{k} d \mu,  
\end{equation}
shows that $B_w$ is a linear factor of $T_f$ \cite[Lemma~4.2.3]{D'AnielloDarjiMaiuriello}, i.e., the following diagram commutes:
\begin{figure}[ht] 
\centering
\begin{tikzcd}
L^p(X)  \arrow[r, "T_f"] \arrow[d, "\Gamma"]
& L^p(X)  \arrow[d, "\Gamma"] \\ \ell^p({\mathbb Z}) \arrow[r,  "B_w"]
&\ell^p({\mathbb Z})
\end{tikzcd} 
\captionsetup{justification=centering,margin=2cm}
\label{Figsemiconj2}
\end{figure} \\

We now state the main result of the paper.
\begin{manualtheorem}M
The operator $T_f$ has Property P, if and only if $B_w$ has Property P where Property P denotes one of: (1) Li-Yorke chaos; (2) hypercyclicity; (3) mixing; (4) chaos; (5) frequent hypercyclicity; (6) expansivity; (7) uniform expansivity; (8) the shadowing property.
\end{manualtheorem}
As the operator $T_f$ of a dissipative composition dynamical system of bounded distortion $(X,{\mathcal B},\mu, f, T_f)$  behaves similarly to shifts, we call these types of operators {\em shift-like.}

We refer the reader to \cite{D'AnielloDarjiMaiuriello,DarjiPires2021} for basic definitions and a fuller discussion of the relationship between weighted shifts, composition operators and the role played by dissipativity and bounded distortion.  

%%%%%%%%%%%%%%%%%%%%%%%%%%%%%%%%
\section{Preliminary Definitions}
%%%%%%%%%%%%%%%%%%%%%%%%%%%%%%%%
In this section we briefly recall basic definitions and terminology. Next to the  definitions, we provide references where one may find fuller and detailed discussions concerning the treated concepts.

Let $\{w_i\}_{i \in \Z}$ be a bounded sequence of positive numbers. A {\em weighted backward shift with weights $\{w_i\}_{i \in \Z}$} is a bounded linear operator $B_w: \ell^p(\Z) \rightarrow \ell^p(\Z)$ defined by
$B_w({\bold x})(i) = w_{i+1}x_{i+1}$.
Moreover, if $\{w_i\}_{i \in \Z}$ is bounded away from zero, then $B_w$ is invertible.

Throughout the paper, we frequently use  the composition operator representation of the weighted backward shift $B_w$. 
 More specifically, we have the following proposition.
 \begin{prop}\label{Tg} Every weighted backward shift $B_w$ is conjugate, by an isometry, to the composition operator $(\Z,{\mathcal P}(\Z), \nu, g, T_g)$,  where \[g(i) = i+1,  \]
\[ \nu(0)=1,  \ \ 
\nu(i)= \begin{cases}
 \dfrac{1}{(w_{1}\cdots w_i)^{p}}, \ \ \ \  \ \  \ \ i > 0 \\
 \left( w_{i+1}\cdots w_0\right)^{p}, \ \ \ \ i < 0.
 \end{cases}\]
 Moreover, for every $i \in \Z$, \[ w_i = \left(\frac{\nu(i-1)}{\nu(i)}\right)^{\frac{1}{p}},\] and when $B_w$ is given by (\ref{wk}), we have \[ \nu(i) =  \dfrac{\mu(f^{i}(W))}{\mu(W)}.\] 
 \end{prop} 
\begin{proof}
That $T_g$ is conjugate to $B_w$ is witnessed by the isometry
\[ {\bf x } \in L^{p}(\Z) \mapsto  \{x_i (\nu(i))^{\frac{1}{p}}\}_{i \in \Z} \in \ell^p(\Z)
\]
\end{proof}

Let $T : X \rightarrow X$ be a bounded linear operator acting on a separable Banach space $X$. 

\begin{defn} The operator $T:X \rightarrow X$ is said to be
\begin{itemize}
\item {\em \cite[Theorem 5]{ BermudezBonillaMartinezPerisJMAA2011} Li-Yorke chaotic } if $T$ admits an {\em irregular vector}, that is, a vector $x \in X$ such that
$ \lowlim_{n \rightarrow \infty} \Vert T^n(x) \Vert =0 \text{ and } \uplim_{n \rightarrow \infty} \Vert T^n(x)\Vert = \infty; $
\item {\em hypercyclic} if $T$  admits a {\em hypercyclic vector}, i.e., if there exists $x$ in $X$ such that $\{T^n(x): n \in \N\}$ is dense in $X$. Moreover if the set of periodic points of $T$ is dense in $X$, then $T$ is said to be {\em chaotic};
\item {\em topologically mixing} if for any pair of non-empty open subsets $U$, $V$ of $X$, there is $k_{0}  \in {\mathbb N}$ such that $T^{k}(U) \cap V \not= \emptyset$, for all $k \geq k_{0}$;

\item {\em frequently hypercyclic} if $T$ admits a {\em frequently hypercyclic vector}, i.e., a vector $x \in X$ such that for each non-empty open subset $U$ of $X$, 
\[\lowlim_{N \rightarrow \infty} \frac{1}{N} \# \{1\leq n \leq N:T^{n}(x) \in U\}>0.\]
\end{itemize}
\end{defn}

The definition of Li-Yorke chaos stated above is equivalent to the standard definition of Li-Yorke chaos  (\cite{BermudezBonillaMartinezPerisJMAA2011, BernardesBonillaMullerPerisETDS2015}). We refer the reader to \cite{BayartMatheronCTM2009, GrosseErdmannMaguillot2011} for further information concerning the other definitions stated above.

Expansivity and the shadowing property play a fundamental role in hyperbolic dynamics. In the context of linear dynamics, they have  alternate formulations equivalent to the ones we use \cite{BernardesCiriloDarjiMessaoudiPujalsJMAA2018}.

\begin{defn}
 An invertible operator $T: X \rightarrow X$ 
\begin{itemize}
    \item is  {\em expansive } if for each $x$ with $\Vert x \Vert =1$, there exists $n \in \mathbb Z$ such that $\Vert T^n x\Vert \geq 2$;
    \item  is  {\em uniformly expansive } if there exists $n \in \mathbb N$ such that \[  z : 
    \Vert z \Vert =1   \Longrightarrow   \Vert T^nz \Vert \geq 2 \text{ or } \Vert T^{-n}z \Vert \geq 2;  \]
    \item  has  the {\em shadowing property} if  there is a constant $K>0$ such that, for every bounded sequence  $\{z_n\}_{n \in \Bbb Z}$ in $X$, there is a sequence $\{y_n\}_{n \in \Bbb Z}$ in $X$ such that \[ \sup_{n \in \Bbb Z} \Vert y_n \Vert \leq K \sup_{n \in \Bbb Z} \Vert z_n \Vert \hspace{0.3 cm}\text{ and } \hspace{0.3 cm} y_{n+1}=Ty_n + z_n, \text{ for all $n \in \Bbb Z$.}\] 
\end{itemize}
\end{defn}

In the definitions of expansivity and uniform expansivity, the number 2 can be replaced by any number $c>1$. For further information on this topic, we refer the reader to \cite{BernardesCiriloDarjiMessaoudiPujalsJMAA2018,BernadesMessaouidPAMS2020, BernardesMessaoudiETDS2020, CiriloGollobitPujals2020}.

%##############################
\section{Properties Preserved  By Factors}
%##############################
In this section, we discuss and state some results when a linear dynamical property is shared by factor maps. Throughout this section,  let $X$ and $Y$ be separable Banach spaces, $S$ and $T$ be bounded, invertible, linear operators on $X$ and $Y$, respectively, and $\Pi: X \rightarrow Y$ be a factor map, i.e., a linear, bounded, surjection for which the following diagram commutes:
\begin{figure}[ht] 
\centering
\begin{tikzcd}
X  \arrow[r, "S"] \arrow[d, "\Pi"]
& X \arrow[d, "\Pi"] \\ Y  \arrow[r,  "T"]
&Y
\end{tikzcd} 
\captionsetup{justification=centering,margin=2cm}
\end{figure} 

%The following follows from the definition.
%\begin{lem}\label{extensionprop}
%If $T$ has the shadowing property, $T$ is expansive or $T$ is uniformly expansive, then $S$ enjoys the same property, respectively.  
%\end{lem}

Next, we state a result which shows when a property carries over to the factor map. 
It is well-known and easy to show that if $S$ is hypercyclic, mixing, chaotic or frequently hypercyclic, then so is $T$ \cite{GrosseErdmannMaguillot2011}. We show that if some extra conditions hold on $\Pi$, then many of the dynamical properties pass from $S$ to $T$. 

In the sequel, given two sequences of positive reals, $\{a_{k}\}_{k \in \mathbb Z}$ and $\{b_{k}\}_{k \in \mathbb Z}$, and a constant $L\geq 1$, 
by $a_{k} \sim_{L} b_{k}$, we mean that, for every $k \in \mathbb Z$, 
\[a_k  \leq Lb_{k} \, \,\, \ \  \& \, \  \ \,\, b_k \leq L a_{k}.\] 
\begin{defn}\label{defSC} \label{defSBS} \label{select}
We say that 
\begin{comment}
\begin{itemize}
\item 
{\em $\Pi$ admits a bounded selector} if 
\[\exists L >0 \text{ s.t., } \forall y \in Y, \exists x \in {\Pi}^{-1}(y) \text{ with }\Vert x \Vert \leq L \Vert y \Vert. \] 
\item
\end{itemize}
\end{comment}
 {\em $\Pi$ admits a strong bounded selector} if
\[\exists L \ge 1 \text{ s.t., } \forall y \in Y, \exists x \in {\Pi}^{-1}(y) \text{ with }\Vert S^n(x) \Vert \sim_L \Vert T^n(y) \Vert. \]

\end{defn}

We have the following result. 

\begin{lem} \label{lemSthenTchaosNEW}  \label{lemSthenT} \label{lemLY2}
For an arbitrary factor map $\Pi$,
\begin{enumerate}[series=myexample]
\item{if $S$ has the shadowing property, also $T$ has the shadowing property.}
\end{enumerate} 
If the factor map $\Pi$ admits a strong bounded selector, the following statements hold:
\begin{enumerate}[resume=myexample]
\item{if $S$ is expansive, then $T$ is expansive;}
\item{if $S$ is uniformly expansive, then $T$ is uniformly expansive;}
\end{enumerate} 
\end{lem}
\begin{proof}
(1). This was proved in \cite[Lemma~ 4.2.2]{D'AnielloDarjiMaiuriello} with the additional hypothesis that $\Pi$ admits a bounded selector, i.e.,
there exists $ L >0$ such that for all $ y \in Y$, there exists $ x \in {\Pi}^{-1}(y)$  with $\Vert x \Vert \leq L \Vert y \Vert$.
We thank Nilson Bernardes for pointing out that the Open Mapping Theorem implies that this is always the case. Indeed, let $\delta >0$ be such that $\overline{ {B}_Y(0,\delta)} \subseteq \Pi(B_X (0,1))$. Then, for any $y \in Y$, we have that there is $x \in \Pi^{-1}(y)$ with $\Vert x \Vert \le \frac{1}{\delta} \Vert y \Vert$.

(2). By hypothesis, there exists a constant $L\ge 1$ such that \[\forall y \in Y, \exists x \in {\Pi}^{-1}(y) \text{ with }\Vert S^n(x) \Vert \sim_L \Vert T^n(y) \Vert.\] Recall that, in general, an operator $S$ is expansive if and only if  $\sup _{n \in \mathbb Z} \Vert S^n  (x)\Vert = \infty$, for every $x \in X\setminus \{0\}$  \cite[Proposition 19]{BernardesCiriloDarjiMessaoudiPujalsJMAA2018}. As $S$ is expansive and $\Vert S^n(x) \Vert \sim_L \Vert T^n(y) \Vert$, then $\sup _{n \in \mathbb Z} \Vert T^n  (y)\Vert = \infty$, i.e., $T$ is expansive.

(3).  Let $K$ be a constant such that $\min \{ \dfrac{K}{L}, \dfrac{K}{L^2} \} \geq 2$. As $S$ is uniformly expansive, let $m \in \mathbb N$ be such that, for each $x \in X$ with $\Vert x \Vert = 1$, we have $\Vert S^m(x)\Vert \geq K$ or $\Vert S^{-m}(x)\Vert \geq K$. Now, let $y \in Y$ with $\Vert y \Vert = 1$. Then, there exists $x \in {\Pi}^{-1}(y)$ such that \[\Vert S^n(x) \Vert \sim_L \Vert T^n(y) \Vert.  \] 
We distinguish two cases. \\
{\bf Case 1.} $\Vert x \Vert =1$.  Then, for some $j\in \{m, -m\}$,
\[\Vert T^j(y)\Vert \geq \dfrac{1}{L}  \Vert S^j(x)\Vert \geq \dfrac{K}{L} \geq 2. \]
{\bf Case 2.} $\Vert x \Vert  \neq 1$. 
As $\left \Vert \dfrac{x}{\Vert x \Vert } \right \Vert = 1$, then, for some $j\in \{m, -m\}$, we have
\begin{eqnarray*}
 \Vert T^j(y) \Vert  \geq  \dfrac{1}{L} {\Vert S^j(x) \Vert} &=&
 \dfrac{1}{L} \left \Vert S^j \left (\frac{x}{\Vert x   \Vert} \right) \right\Vert  \Vert x \Vert \\
  &\geq&  \dfrac{1}{L}  K  \Vert x \Vert .\\
 \end{eqnarray*}
 As $\Pi(x) =y$, by the strong bounded selector condition, we have $\Vert x \Vert \ge \frac{1}{L} \Vert y \Vert = \frac{1}{L}$, implying  
 \begin{eqnarray*}
 \Vert T^j(y) \Vert  
  &\geq&  \dfrac{1}{L}  K  \dfrac{1}{L} = \frac{K}{L^2}
   \geq 2.
 \end{eqnarray*}
\end{proof}
Now, we prove that, in some cases, if a factor enjoys a property, then so does its extension.

\begin{lem} \label{lemLY1}
Suppose that $\Pi$ admits a strong bounded selector. If $T$ is Li-Yorke chaotic, then $S$ is Li-Yorke chaotic. 
\end{lem}

\begin{proof}
By hypothesis, $T$ is Li-Yorke chaotic, that is, $T$ admits an irregular vector $y \in Y$, meaning that there exists $y \in Y$ such that 
\[\lowlim_{n \rightarrow \infty} \Vert T^{n}{y}\Vert = 0 \ \ \ \& \ \ \ \uplim_{n \rightarrow \infty} \Vert T^{n}{y}\Vert = + \infty.\]
By hypothesis, there exists a constant $L\ge 1$ and there exists $ x \in {\Pi}^{-1}(y)$ with $\Vert S^n(x) \Vert \sim_L \Vert T^n(y) \Vert$. 
Clearly, $x$ is an irregular vector for $S$, that is, $S$ is Li-Yorke chaotic.   
\end{proof}

Next, we show that $\Gamma$, as defined in the introduction, satisfies the strong bounded selector condition.

\begin{lem} \label{factorBw} 
The factor map $\Gamma :L^p(X) \rightarrow \ell^p(\Z)$ admits a strong bounded selector. 
\end{lem}
\begin{proof} By a small modification of Lemma 4.2.3 of \cite{D'AnielloDarjiMaiuriello}, we show that $\Gamma$ admits a strong bounded selector. 

 Let ${\bf x}= \{x_k\}_{k \in {\mathbb Z}} \in \ell ^p (\Z) $. 
We let $\varphi$ be as in the proof of Lemma 4.2.3 of \cite{D'AnielloDarjiMaiuriello}, i.e., 
\[\varphi = \sum_{k \in {\mathbb Z}} \frac{x_k}{\mu(f^{k}(W))^{\frac{1}{p}}}  \chi_{f^{k}(W)}.\]
It was shown there that $\Gamma (\varphi) ={\bf{x}}$. Hence, we only need to show that  there exists $L \ge 1$ such that  $\|T_f^n(\varphi)\|_p^p \sim_L \|B_w^n({\bf x})\|_p^p $. In fact, we will show that $\|T_f^n(\varphi)\|_p^p = \|B_w^n({\bf x})\|_p^p $.

Recalling that $w_k=\left (\frac{\mu(f^{k-1}(W))}{\mu(f^k(W))}\right)^{\frac{1}{p}}$, we have
\begin{align*}
\Vert B_w^n ({\bf x}) \Vert_p^p & =\sum_{k \in \mathbb Z} \vert w_{k+1} \cdot \ldots \cdot w_{k+n} x_{k+n} \vert^p = \sum_{k \in \mathbb Z}  \dfrac{\mu(f^k(W))}{\mu(f^{k+n}(W))} \vert x_{k+n} \vert^p.
\end{align*}
Moreover,
\begin{align*}
\Vert T_f^n (\varphi) \Vert_p^p & =\int_X \vert \varphi \vert ^p \circ f^n  d\mu \\
& = \int_{\cup_{h \in \mathbb Z}f^h(W)}  \left \vert \sum_{k \in \mathbb Z}  \frac{x_k}{\mu(f^{k}(W))^{\frac{1}{p}}}  \chi_{f^{k}(W)}\right \vert^p \circ f^n  d\mu  \\
& = \sum_{h \in \mathbb Z} \int_{f^h(W)} \left \vert \sum_{k \in \mathbb Z}  \frac{x_k}{\mu(f^{k}(W))^{\frac{1}{p}}}  \chi_{f^{k}(W)} \right \vert^p \circ f^n  d\mu  \\
& = \sum_{h \in \mathbb Z} \int_{f^h(W)} \left \vert  \frac{x_{h+n}}{\mu(f^{h+n}(W))^{\frac{1}{p}}}  \chi_{f^{h+n}(W)} \right \vert^p \circ f^n d\mu  \\
& = \sum_{h \in \mathbb Z} \frac{\vert x_{h+n} \vert^p }{\mu(f^{h+n}(W))} \int_{f^h(W)}  \chi_{f^{h+n}(W)} \circ f^n d\mu \\
& = \sum_{h \in \mathbb Z} \frac{\vert x_{h+n} \vert^p }{\mu(f^{h+n}(W))} \mu(f^h(W)) \\
&= \Vert B^n_w({\bf x}) \Vert_p^p.
\end{align*}
\end{proof}

%#######################################
\section{Proof for Li-Yorke Chaos}
%######################################

By Lemma \ref{lemLY1} and Lemma \ref{factorBw}, we have that if $B_w$ is Li-Yorke chaotic, then so is $T_f$. We will now show that  the converse is also true. However, we prove a lemma first.
\begin{lem} \label{TNEW1} There exists $L \in \R$ such that for all $n \in \Z$ and $\varphi \in L^p(X)$, we have that 
\[\int_{f^{k}(W)} \vert \varphi \vert^p \circ f^{n} d \mu \sim_{L} \frac{\mu(f^{k}(W))}{\mu( f^{k+n}(W))} \left({\int_{f^{k+n}(W)} \vert \varphi \vert^p d \mu}\right).\]
\end{lem}

\begin{proof}
As the system is of bounded distortion and it is generated by $W$, there exists $L \in \R$, \cite[Proposition~2.6.5] {D'AnielloDarjiMaiuriello}, such that 
\[ \dfrac{1}{L} \frac{\mu(f^{m}(W))}{\mu( f^l(W))}  \leq \dfrac{\mu(f^{m} (B))}{\mu(f^l(B))}   \leq L \frac{\mu(f^{m}(W))}{\mu (f^l(W))}  \tag{$\Diamond \Diamond$},\] for all $ l,m  \in \Z$, and all measurable $B \subseteq W$ with $\mu(B) >0$.

Let $\varphi \in L^{p}(X)$ and $n, k \in \Z$. Let us first prove the result for the case $\varphi|_{f^{k+n}(W)} = \sum_{i=1}^t a_i \chi_ {A_i}$ with $f^{k+n}(W) = \dot{\cup}_{i=1}^t A_i$, and letting  $B_i = f^{-(k+n)}(A_i)$.
Observe that 
\begin{align*}
  \int_{f^{k}(W)} \vert \varphi \vert^p \circ f^n d \mu  & = \int_{f^{k+n}(W)} \vert \varphi \vert^p d\mu (f^{-n}) \\
  & = 
\sum_{i=1}^{t} \vert{a _i}\vert^p {\mu (f^{-n}(A_{i}))}\\ & = \sum_{i=1}^{t} \vert{a _i}\vert^p {\mu (f^{-n}(f^{k+n}(B_{i}))}\\ & = \sum_{i=1}^{t} \vert{a _i}\vert^p {\mu (f^{k}(B_{i}))}. \tag{$\star$}
\end{align*}
Moreover,
\begin{align*}
  \int_{f^{k+n}(W)} \vert \varphi \vert^p  d\mu 
  & = 
\sum_{i=1}^{t} \vert{a _i}\vert^p {\mu (A_{i})}\\ & = \sum_{i=1}^{t} \vert{a _i}\vert^p {\mu (f^{k+n}(B_{i}))}. 
\tag{$\star \star$}
\end{align*}

 Now, applying ($\Diamond \Diamond$) to $B_i$ with $m = k$ and $l = k+n$, we have 
\begin{align*}
    \dfrac{1}{L} \sum_{i=1}^{t}\vert{a _i}\vert^p  \frac{\mu(f^{k}(W))}{\mu( f^{k+n}(W))} \mu(f^{k+n}(B_i)) & \leq  \sum_{i=1}^{t} \vert{a _i}\vert^p \mu(f^k(B_i))  \\ 
    &\leq L \sum_{i=1}^{t}\vert{a _i}\vert^p  \frac{\mu(f^{k}(W))}{\mu( f^{k+n}(W))} \mu(f^{k+n}(B_i)).
\end{align*}
Putting the last inequality together with ($\star$) and ($\star \star$), we obtain 
\begin{align*}
\dfrac{1}{L} \frac{\mu(f^{k}(W))}{\mu( f^{k+n}(W))} \left({\int_{f^{k+n}(W)} \vert \varphi \vert^p d \mu}\right) & \leq \int_{f^{k}(W)} \vert \varphi \vert^p \circ f^{n} d \mu\\
&\leq L \frac{\mu(f^{k}(W))}{\mu( f^{k+n}(W))} \left({\int_{f^{k+n}(W)} \vert \varphi \vert^p d \mu}\right),
\end{align*}
i.e., 
\[\int_{f^{k}(W)} \vert \varphi \vert^p \circ f^{n} d \mu \sim_{L} \frac{\mu(f^{k}(W))}{\mu( f^{k+n}(W))} \left({\int_{f^{k+n}(W)} \vert \varphi \vert^p d \mu}\right).\]

We obtain the proof for an arbitrary $\varphi$ by passing through the limit and applying the Lebesgue dominated convergence theorem.
\end{proof}

{\em Proof: $T_f$ Li-Yorke chaotic $\Rightarrow$ $B_w$ Li-Yorke chaotic.}
By hypothesis $T_f$ admits an irregular vector $\varphi \in L^p(X)$, meaning that there exists $\varphi \in L^p(X)$ such that 
\[\lowlim_{n \rightarrow \infty} \Vert T_f^{n}{\varphi}\Vert_p = 0 \ \ \ \& \ \ \ \uplim_{n \rightarrow \infty} \Vert T_f^{n}{\varphi}\Vert_p = + \infty.\]
We want to apply Corollary 1.5 in \cite{BernardesDarjiPiresMM2020} to show that $T_g$, as defined in Proposition \ref{Tg}, is Li-Yorke chaotic, or equivalently that $B_w$ is Li-Yorke chaotic.
Note that, using Lemma \ref{TNEW1}, we have, for each $n,k \in \Z$,
\begin{align*}
\Vert T_f^n \varphi \Vert_p^p= \int_X \vert \varphi \vert ^p \circ f^n d\mu & \geq    \int_{f^k(W)} \vert \varphi \vert ^p \circ f^n d\mu \\
& \geq \frac{1}{L} \frac{\mu(f^{k}(W))}{\mu( f^{k+n}(W))} \left({\int_{f^{k+n}(W)} \vert \varphi \vert^p d \mu}\right),
\end{align*}
and, in particular, for $k=-n$,  
\[\Vert T_f^n \varphi \Vert_p^p \geq \frac{1}{L} \frac{\mu(f^{-n}(W))}{\mu(W)} \left({\int_{W} \vert \varphi \vert^p d \mu}\right).\]
As $\lowlim_{n \rightarrow \infty} \Vert T_f^{n}{\varphi}\Vert_p = 0,$ then  $\lowlim_{n \rightarrow \infty} \mu(f^{-n}(W)) = 0,$ and, hence, \[\lowlim_{n \rightarrow \infty} \nu(g^{-n}(0))= \lowlim_{n \rightarrow \infty} \nu(-n)=\lowlim_{n \rightarrow \infty} \dfrac{\mu(f^{-n}(W))}{\mu(W)}=0,\]
that is, Condition $(a)$ of Corollary 1.5 in \cite{BernardesDarjiPiresMM2020} is satisfied, i.e. $\lowlim_{n \rightarrow -\infty} \nu(n)=0$. \\
Now, we show that Condition $(b)$ of Corollary 1.5 in \cite{BernardesDarjiPiresMM2020} is satisfied, i.e., \[\sup \left \{ \dfrac{\nu(h)}{\nu(h+n)}, h \in \Z, n \in \N \right \}= \infty, \] or, equivalently,
$\sup \left \{ \dfrac{\mu(f^h(W))}{\mu(f^{h+n}(W))}, h \in \Z, n \in \N \right \}= \infty.$
To obtain a contradiction, assume that $\dfrac{\mu(f^h(W))}{\mu(f^{h+n}(W))}$ is bounded above by a constant $H$. Hence, from Lemma \ref{TNEW1}, it follows
\[ \int_{f^h(W)} \vert \varphi \vert ^p \circ f^n d\mu \leq LH \int_{f^{h+n}(W)} \vert \varphi \vert ^p d\mu \ \ \ \text{ for every } h \in \Z, n \in \N ,\]
implying 
\[ \int_{X} \vert \varphi \vert ^p \circ f^n d\mu \leq LH \int_{X} \vert \varphi \vert ^p d\mu, \text{ for every } n \in \N, \] contradicting the fact that $\varphi$ is an irregular vector.

%#######################################
\section{Proof for Hypercyclicity and Mixing}
%######################################

If $T_f$ is hypercyclic or mixing, then, using Propositions 1.13 and 1.40 of \cite{GrosseErdmannMaguillot2011}, it follows  that $B_w$ is hypercyclic or mixing, respectively.
To see the converse, we need the following two conditions.

 \begin{prop} \label{newconditiontransmix}
 The following statements hold.
 \begin{enumerate}
     \item If, for each $\epsilon >0$ and  $N \in \mathbb N$, there exists $k \geq 1$ such that, 
     for $h \in \{k,-k\}$,  \[ \mu(f^h(\cup_{\vert j \vert \leq N}f^j(W)))<\epsilon,\] then $T_f$ is hypercyclic.
     \item If, for each $\epsilon >0$ and  $N \in \mathbb N$, there exists $k_0 \in \mathbb N$, such that, for each $k \geq k_0$ and 
     $h \in \{ k,-k\}$,\[ \mu(f^h(\cup_{\vert j \vert \leq N}f^j(W)))<\epsilon,\]
then $T_f$ is mixing.
 \end{enumerate}

\end{prop}
\begin{proof} 
We will use characterizations of hypercyclicity and mixing given in \cite[Theorems 1.1 and 1.2]{BayartDarjiPiresJMAA2018}.

(1). Let $\epsilon >0$ and $B \in \cal B$, with $0<\mu(B) < \infty$. Let $N \in \mathbb N$ be so large 
that \[ \mu\left( B \setminus (\bigcup_{j=-N}^N( B \cap f^j(W)))\right) < \epsilon.\]
Define $B' =\cup_{j=-N}^N (B \cap f^j(W))$. Then, $\mu(B \setminus B') < \epsilon$. 
By hypothesis,  there exists $k \geq 1$ such that, for $h \in \{k,-k\}$, \[ \mu(f^h(\cup_{\vert j \vert \leq N}f^j(W)))<\epsilon.\]
Hence, 
\begin{eqnarray*}
\mu(f^h(B'))\le \mu(f^h (\cup_{j=-N}^N  f^j(W))) < \epsilon.  
\end{eqnarray*}

We have just shown that,  
for all $\epsilon >0$ and for all $B \in \cal B$ with $0<\mu(B) < \infty$, there exist 
$B' \subset B$ and $k \geq 1$ such that 
\[ \mu(B \setminus B')<\epsilon, \,\, \mu(f^{-k}(B'))<\epsilon \text{ and }\,\,  \mu(f^{k}(B'))<\epsilon. \] Now, it follows from 
Theorem 1.1 in  \cite{BayartDarjiPiresJMAA2018} that $T_f$ is hypercyclic. \\

(2). Let $\epsilon$, $B$, $N$, and $B'$ be as in part (1). Let $k_0$ be the integer guaranteed by the  hypothesis. Letting, for every $k > k_0$, $B_k =B'$, we have, for all $h \in \{k, -k\}$,

\[ \mu(B \setminus B_k)<\epsilon \ \text{ and }  \  \mu(f^{h}(B_k))<\epsilon.  \]

Now, by Theorem 1.2 of \cite{BayartDarjiPiresJMAA2018}, we have that $T_f$ is mixing.

\end{proof}

{\em Proof: $B_w$ hypercyclic $\Rightarrow$ $T_f$ is hypercyclic.}
The operator $T_g$ is hypercyclic by Proposition \ref{Tg}. We will show that hypothesis (1) of Proposition~\ref{newconditiontransmix} is satisfied by $T_f$, so that it is hypercyclic. To this end, let $\epsilon >0$ and $N \in \N$.
 Applying Theorem 1.1 of \cite{BayartDarjiPiresJMAA2018} to $B=\cup_{j=-N}^{N}  g^{j}(0)$ and  $\tilde{\epsilon}= 
 \min_{-N \le j \le N} \frac{1}{2} \{\dfrac{\epsilon}{\mu(W)},  \nu(g^{j}(0))\}$, we have that there exist $B' \subseteq B$ and $k \geq 1$ such that, for $h \in \{k,-k\}$, \[ \nu(B \setminus B')<\tilde{\epsilon}, \,\, \nu(g^{h}(B'))<\tilde{\epsilon}.
\]
 By our choice of $\tilde{\epsilon}$, we have that $B' = B = \cup_{j=-N}^N  g^{j}(0)$. In particular, we have, for $h \in \{k,-k\}$, \[\nu(g^{h}(\cup_{j=-N}^N  g^{j}(0)))= \sum_{j=-N}^N\nu(g^{h+j}(0)) < \tilde{\epsilon} <\dfrac{\epsilon}{\mu(W). } \tag{$\heartsuit$}\] 

Substituting  
\[\nu(g^{n}(0))=\nu(n)= \dfrac{\mu(f^{n}(W))}{\mu(W)}\]
in ($\heartsuit$), we obtain
\[ \sum_{j=-N}^N\nu(g^{h+j}(0)) =  \sum_{j=-N}^N  \frac{\mu(f^{h+j}(W))}{\mu(W)}  =\frac{\mu( f^{h}( \cup_{j=-N}^{N} f^{j}(W) ))}{\mu(W)}  < \dfrac{\epsilon}{\mu(W)},\] 
i.e.,
\begin{eqnarray*}
\mu( f^{h}( \cup_{j=-N}^{N} f^{j}(W) )) < \epsilon.
 \end{eqnarray*}
By Proposition \ref{newconditiontransmix}-(1),  $T_f$ is hypercyclic. \\

{\em Proof: $B_w$ mixing $\Rightarrow$ $T_f$ is mixing.}
We will show that hypothesis of (2) of Proposition~\ref{newconditiontransmix} is satisfied, so  that $T_f$ is mixing. We proceed as in the proof of hypercyclicity.
Let $\epsilon >0$ and  $N \in \N$. Let $B$ and ${\tilde{\epsilon}}$ as in the proof of hypercyclicity. 
As $B_w$ is mixing, applying Theorem 1.2 of \cite{BayartDarjiPiresJMAA2018} to $B$ and ${\tilde{\epsilon}}$, there exist $k_0 \geq 1$ and a sequence $\{B_k\}$ such that, for each $k \geq k_0$, $h \in \{-k,k\}$, 
\[ \nu(B \setminus B_k)<\tilde{\epsilon}, \,\, \nu(g^{h}(B_k))<\tilde{\epsilon}.\]
Then, by the choice of $\tilde{\epsilon}$, it must be $B_k=B$ for each $k \geq k_0$.  Hence, for each $k \geq k_0$ and $h \in \{ k,-k\}$, we have 
\[\nu(g^{h}(\cup_{j=-N}^N  g^{j}(0)))= \sum_{j=-N}^N\nu(g^{h+j}(0)) < \tilde{\epsilon} <\dfrac{\epsilon}{\mu(W)}.\] 
Hence, ($\heartsuit$) is satisfied for each $k \geq k_0$. Now, proceeding as in the proof of hypercyclicity, we have that the hypothesis of (2) of Proposition~\ref{newconditiontransmix} is satisfied.

%#######################################
\section{Proof for Chaos and Frequent Hypercyclicity}
%######################################

It follows from Theorem~3.7 \cite{DarjiPires2021} that $T_f$ is chaotic if and only if $T_f$ is frequently hypercyclic. Hence, it suffices to show the result for frequently hypercyclic operators.

If $T_f$ is frequently hypercyclic, then that $B_w$ is frequently hypercyclic follows from the fact that frequent hypercyclicity is preserved by factor maps  \cite[Proposition 9.4]{GrosseErdmannMaguillot2011}. To see the converse, assume that $B_w$, equivalently $T_g$, is frequently hypercyclic. Note that $\nu$ is atomic and $g$ is ergodic on $\Z$. Applying Corollary~3.9 of \cite{DarjiPires2021}, we have $\nu(\Z) < \infty$. However, $\nu(\Z) = \dfrac{\mu(X)}{\mu(W)}$, implying $\mu(X) < \infty$. Now, applying  Theorem~3.3 of \cite{DarjiPires2021}, we have that $T_f$ is frequently hypercyclic, completing the proof.

%#######################################
\section{Proof for Expansivity and Uniform Expansivity}
%######################################
By Lemma \ref{lemSthenT} and Lemma \ref{factorBw}, it follows that if $T_f$ is expansive or uniform expansive, then so is $B_w$. Now we will prove the converse. 

{\em Proof: $B_w$ expansive $\Rightarrow$ $T_f$ expansive.}
Assume that $B_w$, equivalently $T_g$, is expansive, i.e. $\sup_{n \in \Z}\Vert T_g^n \varphi \Vert_p= \infty,$ \cite[Proposition 19]{BernardesCiriloDarjiMessaoudiPujalsJMAA2018} for each $\varphi \in L^p(\Z)\setminus \{0\}.$ Hence, taking $\varphi =\chi_{\{0\}}$, it follows that
\begin{eqnarray*}
     \infty = \sup_{n \in \Z}\Vert T_g^n \varphi \Vert_p^p= \sup_{n \in \Z} \int_{\Z} \chi_{\{0\}} \circ g^n d\nu &=& \sup_{n \in \Z} \nu(g^{-n}(0)) \\
     &=& \sup_{n \in \Z} \nu(-n) = \sup_{n \in \Z} \dfrac{\mu(f^{-n}(W))}{\mu(W)}, 
     \end{eqnarray*}
i.e., $\sup_{n \in \Z} \mu(f^n(W))= \infty$. This will imply that $T_f$ is expansive. In fact, using Lemma \ref{TNEW1} with $k=-n$, for each $\varphi \in L^p(X) \setminus \{0\}$, we have 
\begin{align*}
    \Vert T_f^n \varphi \Vert_p^p= \int_X \vert \varphi \vert ^p \circ f^n d\mu & \geq    \int_{f^{-n}(W)} \vert \varphi \vert ^p \circ f^n d\mu \\
& \geq \dfrac{1}{L} \frac{\mu(f^{-n}(W))}{\mu(W)} \left({\int_{W} \vert \varphi \vert^p d \mu}\right).
\end{align*}
Now,
\[\sup_{n \in \Z} \Vert T_f^n \varphi \Vert_p^p \geq \dfrac{1}{L} \frac{{\int_{W} \vert \varphi \vert^p d \mu}}{\mu(W)}  \sup_{n \in \Z}(\mu(f^{-n}(W))) = \infty,\]
i.e., $T_f$ is expansive.

{\em Proof: $B_w$ uniform expansive $\Rightarrow$ $T_f$ uniform expansive.} We use the definition of uniform expansivity. Let $L$ be as in the statement of Lemma~\ref{TNEW1}.  Note that $L >1$. By the definition of uniform expansivity, there exits $n \in \N$ such that for all $\varphi \in \ell^p(\Z)$ with $\Vert \varphi \Vert _p = 1$, one of the following holds: \[ \Vert T_g^n \varphi \Vert_p^p \geq 4L \text{ or } \Vert T_g^{-n} \varphi \Vert_p^p \geq 4L.\]
Taking $\varphi = \frac{\chi_{\{i\}}}{ \nu(i)^{\frac{1}{p}}}$, we obtain
\[\Vert T_g^n \varphi \Vert_p^p= \dfrac{\nu(g^{-n}(i))}{\nu(i)} \ \ \ \ \& \ \ \ \  \Vert T_g^{-n} \varphi \Vert_p^p= \dfrac{\nu(g^{n}(i))}{\nu(i)}.\]
Recalling that $\dfrac{\nu(g^{k}(i))}{\nu(i)} =\dfrac{\mu(f^{k+i}(W))}{\mu(f^i(W))},$  we have that, for each $i \in \Z$, either
\[ \dfrac{\mu(f^{n+i}(W))}{\mu(f^i(W))}  \ge 4L \ \ \ \text{or} \ \ \ \dfrac{\mu(f^{-n+i}(W))}{\mu(f^i(W))}  \ge 4L.
\] 
Let $N_+$ be the set of $i \in \Z$ such that the first holds and $N_-$ be the set of $i \in \Z$ where the second holds. Clearly, $\Z = N_+ \cup N_-$. 

Let $\varphi \in L^p(X) $ with $\Vert \varphi \Vert_p=1$. It will suffice to show that either
\[ \Vert T_f^n  \varphi \Vert _p^p > 2  \ \ \ \text{or } \ \ \ \ \Vert T_f^{-n}  \varphi \Vert _p^p > 2.
\] As $\Vert \varphi \Vert_p=1$, then either 
\[ \sum_{i \in \N_+} {\int_{f^i(W)} \vert \varphi \vert^p d \mu} \ge  1/2   \ \ \ \ \text{or}  \ \ \sum_{i \in \N_-} {\int_{f^i(W)} \vert \varphi \vert^p d \mu} \ge  1/2. \] 
Without loss of generality, assume that the first holds. Now, 
using Lemma \ref{TNEW1}, we have
\begin{align*}
    \Vert T_f^n \varphi \Vert_p^p= \int_X \vert \varphi \vert ^p \circ f^n d\mu & \geq \sum_{i \in N_+}    \int_{f^{-n+i}(W)} \vert \varphi \vert ^p \circ f^n d\mu \\
& \geq \dfrac{1}{L}\sum_{i \in N_+}   \frac{\mu(f^{-n+i}(W))}{\mu(f^i(W))} \left(\int_{f^i(W)} \vert \varphi \vert^p d \mu \right)\\
& \geq \frac{1}{L} 4L \sum_{i \in N_+}    \left(\int_{f^i(W)} \vert \varphi \vert^p d \mu \right)\\
& \geq 2.
\end{align*}

%######################################
\section{Proof for the Shadowing Property}
%######################################
By Corollary SC of  \cite{D'AnielloDarjiMaiuriello}, we have that $T_f$ has the shadowing property if and only if 
one of Conditions $\hc{}$, $\hd{}$ or $\gh{}$ in it is satisfied.

By Theorem~18 of  \cite{BernardesMessaoudiETDS2020}, we have that $B_w$ has the shadowing property if and only if one of Conditions (A), (B) or (C) in it is satisfied.

As
 \[ w_{k} =  \left( \frac{\mu(f^{k-1}(W))}{\mu(f^{k}(W))}\right)^{\frac{1}{p}},\] 
we have that Condition~$\hc{}$ is equivalent to Condition (A), $\hd{}$ is equivalent to Condition (B) and  $\gh{}$ is equivalent to Condition (C), implying that $T_f$ has the shadowing property if and only if $B_w$ does.  

\section*{Acknowledgment}
We would like to thank Nilson Bernardes for his valuable comments which improved the article. \\
This research has been partially supported by the INdAM group GNAMPA ``Gruppo Nazionale per l’Analisi Matematica, la Probabilità e le loro Applicazioni'', and the project Vain-Hopes within the program VALERE; and partially accomplished within the UMI group TAA ``Approximation Theory and Applications''.

%------BIBLIOGRAPHY--------
\bibliography{biblio}

\begin{thebibliography}{10}

\bibitem{BayartDarjiPiresJMAA2018}
{\sc F.~Bayart, U.~B. Darji, and B.~Pires}, {\em Topological transitivity and
  mixing of composition operators}, J. Math. Anal. Appl., 465 (2018),
  pp.~125--139.

\bibitem{BayartMatheronCTM2009}
{\sc F.~Bayart and E.~Matheron}, {\em Dynamics of linear operators}, vol.~179
  of Cambridge Tracts in Mathematics, Cambridge University Press, Cambridge,
  2009.

\bibitem{BayartRuzsa}
{\sc F.~Bayart and I.~Z. Rusza}, {\em Difference sets and frequently
  hypercyclic weighted shifts}, Ergodic Theory Dynam. Systems, 35 (2015),
  pp.~691--709.

\bibitem{BermudezBonillaMartinezPerisJMAA2011}
{\sc T.~Berm\'{u}dez, A.~Bonilla, F.~Mart\'{\i}nez-Gim\'{e}nez, and A.~Peris},
  {\em Li-{Y}orke and distributionally chaotic operators}, J. Math. Anal.
  Appl., 373 (2011), pp.~83--93.

\bibitem{BernardesBonillaMullerPerisETDS2015}
{\sc N.~C. Bernardes, Jr., A.~Bonilla, V.~M\"{u}ller, and A.~Peris}, {\em
  Li-{Y}orke chaos in linear dynamics}, Ergodic Theory Dynam. Systems, 35
  (2015), pp.~1723--1745.

\bibitem{BernardesCiriloDarjiMessaoudiPujalsJMAA2018}
{\sc N.~C. Bernardes, Jr., P.~R. Cirilo, U.~B. Darji, A.~Messaoudi, and E.~R.
  Pujals}, {\em Expansivity and shadowing in linear dynamics}, J. Math. Anal.
  Appl., 461 (2018), pp.~796--816.

\bibitem{BernardesDarjiPiresMM2020}
{\sc N.~C. Bernardes, Jr., U.~B. Darji, and B.~Pires}, {\em Li-{Y}orke chaos
  for composition operators on {$L^p$}-spaces}, Monatsh. Math., 191 (2020),
  pp.~13--35.

\bibitem{BernadesMessaouidPAMS2020}
{\sc N.~C. Bernardes, Jr. and A.~Messaoudi}, {\em A generalized
  {G}robman-{H}artman theorem}, Proc. Amer. Math. Soc., 148 (2020),
  pp.~4351--4360.

\bibitem{BernardesMessaoudiETDS2020}
\leavevmode\vrule height 2pt depth -1.6pt width 23pt, {\em Shadowing and
  structural stability for operators}, Ergodic Theory Dynam. Systems, 41
  (2021), pp.~961--980.

\bibitem{BDDPOdometers}
{\sc D.~Bongiorno, E.~D'Aniello, U.~B. Darji, and L.~Di~Piazza}, {\em Linear
  dynamics induced by odometers}, Proc. Amer. Math. Soc., 150 (2022),
  pp.~2823--2837.

\bibitem{CiriloGollobitPujals2020}
{\sc P.~Cirilo, B.~Gollobit, and E.~Pujals}, {\em Dynamics of generalized
  hyperbolic linear operators}, Adv. Math., 387, Article ID: 107830 (2021),
  pp.~1--37.

\bibitem{D'AnielloDarjiMaiuriello}
{\sc E.~D'Aniello, U.~B. Darji, and M.~Maiuriello}, {\em Generalized
  hyperbolicity and shadowing in ${L}^p$ spaces}, J. Differential Equations,
  298 (2021), pp.~68--94.

\bibitem{DarjiPires2021}
{\sc U.~B. Darji and B.~Pires}, {\em Chaos and frequent hypercyclicity for
  composition operators}, Proc. Edinburgh Math. Soc.,  (2021), p.~1–19.

\bibitem{GrosseErdmannMaguillot2011}
{\sc K.-G. Grosse-Erdmann and A.~Peris~Manguillot}, {\em Linear chaos},
  Universitext, Springer, London, 2011.

\end{thebibliography}
\bibliographystyle{siam}
\Addresses
\end{document}